\documentclass[a4paper,12pt,reqno]{amsart}
\usepackage[T1]{fontenc}
\usepackage[utf8]{inputenc}
\usepackage{lmodern}
\usepackage{amsmath, amsthm, amssymb,amscd, mathrsfs, amsfonts, mathtools}
\usepackage{hyperref}
\usepackage{verbatim}
\usepackage{graphicx}
\usepackage{float}
\usepackage{booktabs}
\usepackage{euler}
\usepackage{times}
\usepackage[all]{xy}
\usepackage{todonotes}
\usepackage{xcolor}
\usepackage{tikz}
\usetikzlibrary{hobby,knots,graphs,graphs.standard,arrows.meta}
\usetikzlibrary{decorations.pathreplacing}

\usepackage[lite]{amsrefs}

\renewcommand{\PrintDOI}[1]{\href{http://dx.doi.org/\detokenize{#1}}{doi: \detokenize{#1}}%
  \IfEmptyBibField{pages}{, (to appear in print)}{}}

\def\commutatif{\ar@{}[rd]|{\circlearrowleft}}

\newcommand{\eq}[1][r]
   {\ar@<-3pt>@{-}[#1]
    \ar@<-1pt>@{}[#1]|<{}="gauche"
    \ar@<+0pt>@{}[#1]|-{}="milieu"
    \ar@<+1pt>@{}[#1]|>{}="droite"
    \ar@/^2pt/@{-}"gauche";"milieu"
    \ar@/_2pt/@{-}"milieu";"droite"}

\def\dar[#1]{\ar@<2pt>[#1]\ar@<-2pt>[#1]}
  \entrymodifiers={!!<0pt,0.7ex>+} %%%%%%%%%%%%%%%%%

\newcommand{\bigon}[4][r]{% %%%%% bigons
    \ar@/^1pc/[#1]^{#2}_*=<0.3pt>{}="HAUT"
    \ar@/_1pc/[#1]_{#3}^*=<0.3pt>{}="BAS"
    \ar@{=>} "HAUT";"BAS" ^{#4}
  }

\newcommand{\bigons}[6][r]{  %%%%% Vertical composition of bigons
    \ar@/^2pc/[#1]^{#2}_*=<0.3pt>{}="HAUT"
    \ar@{}    [#1]     ^*=<0.3pt>{}="MILIEUHAUT"
                       _*=<0.3pt>{}="MILIEUBAS"
    \ar[#1]_(0.3){#3}                  
    \ar@/_2pc/[#1]_{#4}^*=<0.3pt>{}="BAS"
    \ar@{=>} "HAUT";"MILIEUHAUT" ^{#5}
    \ar@{=>} "MILIEUBAS";"BAS" ^{#6}
  }

\newtheorem{thm}{Theorem}[section]

\theoremstyle{definition}
\newtheorem{df}[thm]{Definition}

\theoremstyle{remark}

\newtheorem{ex}[thm]{Example}

\allowdisplaybreaks

\newcommand\Hom{\operatorname{Hom}}

\newcommand{\Z}{\mathbb{Z}}

\newcommand\rTo{\longrightarrow}

\let\cal\mathcal

\hypersetup{
  colorlinks = true,
  urlcolor = blue,
  linkcolor = blue,
  citecolor = red,
  pdfauthor = {Elhamdadi, M., Liu, M.} 
  pdfkeywords = {Quandles, quasi-trivial quandles and biquandles, cocycle invariants },
  pdftitle = {Elhamdadi, M., Liu, M., - Quasi-trivial Quandles and Biquandles,  Cocycle Enhancements and Link-Homotopy of Pretzel links},
  pdfsubject = {quandles},
  pdfpagemode = UseNone
}

\title[Quandle Coloring Quivers of general Torus links]{Quandle Coloring Quivers of general Torus links by dihedral quandles}

\author{Mohamed Elhamdadi} 
\address{Department of Mathematics, 
University of South Florida, Tampa, FL 33620} 
\email{emohamed@usf.edu}

\author{Brooke Jones} 
\address{Department of Mathematics, 
University of South Florida, Tampa, FL 33620} 
\email{brookejones1@usf.edu} 
\author{Minghui Liu} 

\address{Florida College, Temple Terrace, FL 33617} 
\email{lium@floridacollege.edu}

\begin{document}

\maketitle

\begin{abstract}
%We investigate 
We completely characterize the coloring quivers of general torus links by dihedral quandles by first exhausting all possible numbers of colorings, followed by determining the interconnections between colorings in each case.
The quiver is obtained as function of the number of colorings.
The quiver always contains complete subgraphs, in particular a complete subgraph corresponding to the trivial colorings, but the total number of subgraphs in the quiver and the weights of their edges varies depending on the number of colorings. 
\end{abstract}

%\tableofcontents

\section{Introduction}
 The concept of a quiver and its representations was introduced in \cite{Gabriel} by P. Gabriel in 1972.  A quiver is an oriented graph; precisely, it is a pair $Q=(V,E)$, where $V$ is a finite set of vertices and $E$ is a finite set of arrows between them.  Quivers have been used in many areas of mathematics such as representation theory of finite dimensional algebras, ring theory (Gabriel's theorem), Lie algebras (Dynkin diagrams) and quantum groups among others.
A quiver structure on the set of quandle colorings of an oriented knot was introduced in \cite{CN} where some enhancements of the counting invariant were obtained from the quiver structure.  In \cite{BC}, quandle coloring quivers of $(2, q)$-torus links by dihedral quandles were studied.  The case of $(3, q)$-torus links was investigated in \cite{ZL}.  The main goal of this paper is to generalize the results of \cites{BC,ZL} to  any torus link $T(p,q)$. To give an overview of the process, we decided to first cover the cases T(5,q) and T(7,q) before the general case T(p,q).  In investigating dihedral quandle coloring quivers of general torus links, we first exhaust all possible numbers of colorings of the torus links by dihedral quandles $R_n$. The quiver is obtained as function of the number of colorings. 
Precisely, (1) if the number of colorings is $n$ then the quiver is the complete graph on $n$ vertices with each edge having weight $n$. (2) if the number of colorings is $pn$ then the quiver graph is composed of two complete graphs, one graph on $n$ vertices and the other on $(p-1)n$ vertices, such that there is a directed edge with weight $\frac{n}{p}$ from one graph to the other. (3) if the number of colorings is $2^{p-1}n$, then the quiver contains $2^{p-1}$ complete subgraphs of order $n$. (4) if the number of colorings is $n^p$, we obtain $\frac{n^p-n}{n(n-1)}$ disjoint subgraphs with vertices each having edges to a complete subgraph on $n$ vertices. While some of the complete subgraphs of the quivers are of different orders, we also find that the weights of the edges between them vary depending on the number of colorings.

The paper is organized as follows.  In Section~\ref{Rev} we review the basics of quandles and link colorings. Section~\ref{Ex} gives the quandle coloring quivers of the torus links $T(5,q)$ and $T(7,q)$.  In Section~\ref{Gen}, we give the main result of the paper in which we characterize the quandle coloring quivers of general torus links $T(p,q)$, where $p$ is prime.

\section{A review of Quandles and link colorings}\label{Rev}
In this section, we collect some basics about quandles that we will need through the paper.  We begin with the following definition taken from \cites{EN, Joyce, Matveev}.

\begin{df} \label{quandledef}
	A {\it quandle} is a set $X$ with a binary operation 

$*:  X\times X  \rTo X$ %which maps $(x,y)$ to $x* y$, 
satisfying the following three axioms: 
\begin{itemize}

\item[(i)]({\it right distributivity}) for all $x,y,z\in X$, we have $(x* y)* z=(x* z)* (y* z)$;
\item[(ii)] ({\it invertibility}) for all $x\in X$, the map $R_x:X\rTo X$ sending $y$ to $y* x$ is a bijection;
%there is a unique $z\in X$ such that $y=z* x$;
\item[(iii)] ({\it idempotency}) for all $x\in X$, $x* x=x.$
\end{itemize}
\end{df}
Since axiom (ii) states that for each $y \in X$, the map $R_y$ is invertible, we will denote $R_y^{-1}(x)$ by $x\Bar{*}y$. When the binary operation satisfies $(x* y)* y=x$ for all $x, y \in X$, then the quandles is called a {\it kei} (or involutive quandle).  
%Observe that property (i) also reads that for any fixed element $x\in X$, the map $R_x:X\rTo X$ sending $y$ to $y* x$ is a bijection. Also, notice .
Let $n$ be a positive integer. For $x, y \in \Z_n$ (integers modulo $n$), define $x * y = 2y-x \pmod{n}$. Then the operation $*$ defines a quandle structure  called {\it dihedral quandle} and is denoted $R_n$.

Let $D$ be a diagram of an oriented link $K$ and let $X$ be a quandle. A {\it coloring} of $D$ by the quandle $X$ is a map ${\cal C}: {\cal A} \rightarrow X$
 from the set of arcs ${\cal A}$ of the diagram $D$ to $X$ such that the
 image of the map satisfies the relation depicted in
 Figure~\ref{fig:posneg_crossing_quandle} at each crossing. 
%A coloring of $D$ by $X$ is defined as follows.  We color the arcs of $D$ according to Figure~\ref{fig:posneg_crossing_quandle}.  
For more on colorings of knots by quandles the reader can consult for example \cites{EN, Joyce, Matveev}.

\begin{figure}[H]
    \centering
    \includegraphics[width=3.7in]{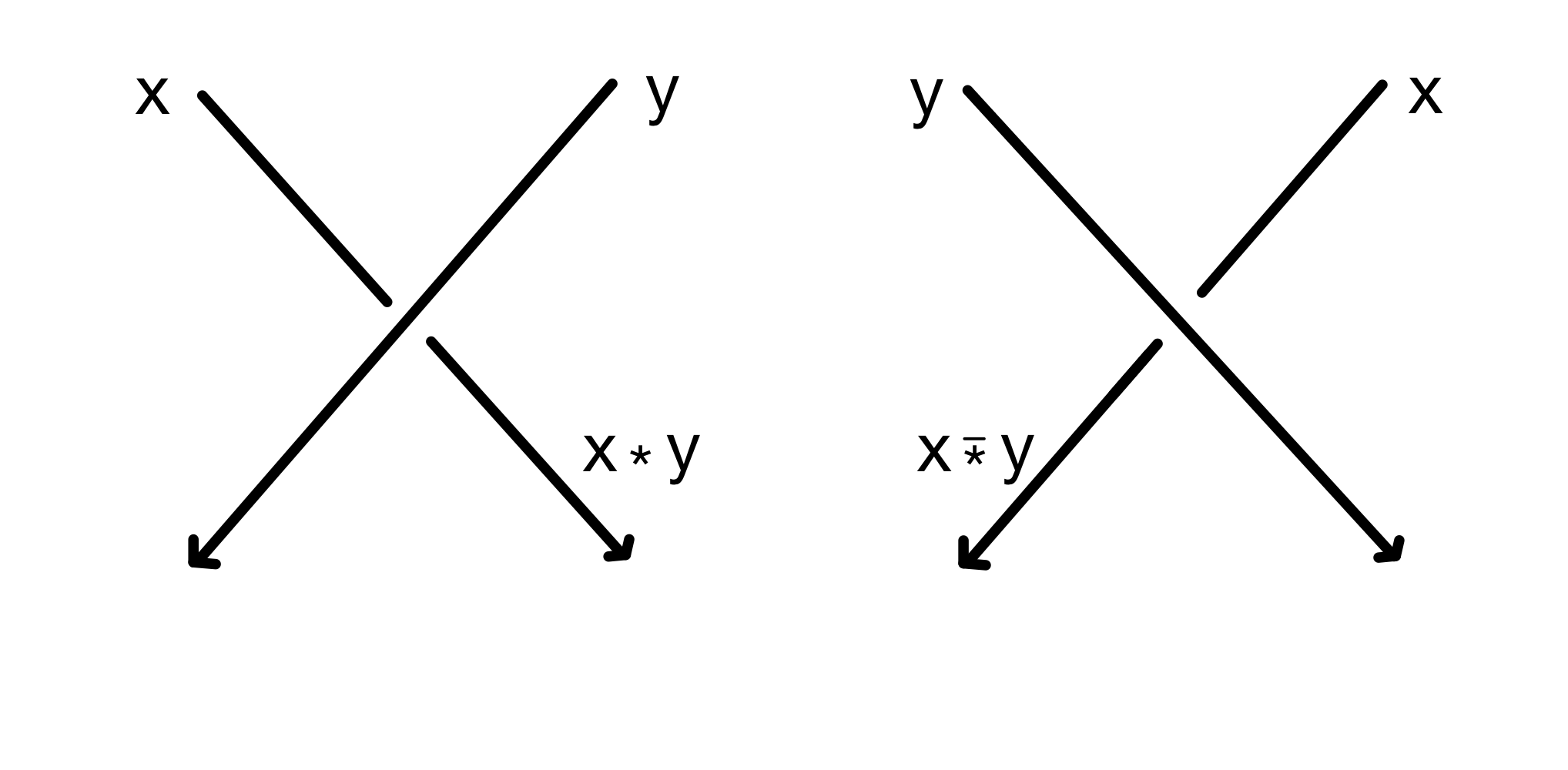}
  % \begin{tikzpicture}[scale=1.3]
    % Draw the over arc
 %   \draw[->, thick] (-1,2) -- (-3,0);
    % Draw the under arc with a break to demonstrate the undercrossing
 %   \draw[->, thick] (-1.9,0.9) -- (-1,0);
%    \draw[thick] (-3,2) -- (-2.1,1.1);
    % Add the crossing information
%    \node at (-2.7,1.9) {$x$};
%    \node at (-1.3,1.9) {$y$};
 %   \node at (-1,.43) {$x*y$}; % The quandle operation

% Draw the over arc
%    \draw[->, thick] (1.9,.9) -- (1,0);
    % Draw the under arc with a break to demonstrate the undercrossing
%    \draw[->, thick] (1,2) -- (3,0);
%    \draw[thick] (3,2) -- (2.1,1.1);
    % Add the crossing information
%    \node at (1.3,1.9) {$y$};
%    \node at (2.7,1.9) {$x$};
%    \node at (1,.43) {$x\overline{*}y$}; % The quandle operation    
%    \end{tikzpicture}
    %\caption{The quandle operation $*$ corresponding to positive crossings, and $\overline{*}$ corresponding to negative crossings.}
    \caption{Colorings of arcs at positive and negative crossings}
    \label{fig:posneg_crossing_quandle}
\end{figure}

For any link $K$, there is an associated fundamental quandle $Q(K)$ of the knot $K$ given by presentation with generators corresponding to the arcs of a diagram $D$ of $K$ and quandle relations at crossings of $D$ as in Figure~\ref{fig:posneg_crossing_quandle}. The set of
colorings of a knot $K$ by a quandle $X$, then, is in
one-to-one correspondence with the set of quandle homomorphisms from
the fundamental quandle $Q(K)$ of $K$ to $X$.
For example, the fundamental quandle of \emph{figure eight} knot can be obtained from Figure~\ref{fig:figure8knot}.  It is given by $Q(4_1)=\langle x,y,z,w; \;x \Bar{*}z=y, z \Bar{*}x=w, w*y=x,y*w=z  \rangle= \langle x,y,z; \; y*z=x, (z \Bar{*}x)*y=x, y*(z \Bar{*}x)=z\rangle $.  
\vspace{-.3cm}
\begin{figure}[h]
    \centering
    \includegraphics[width=2.2in, height=2.2in]{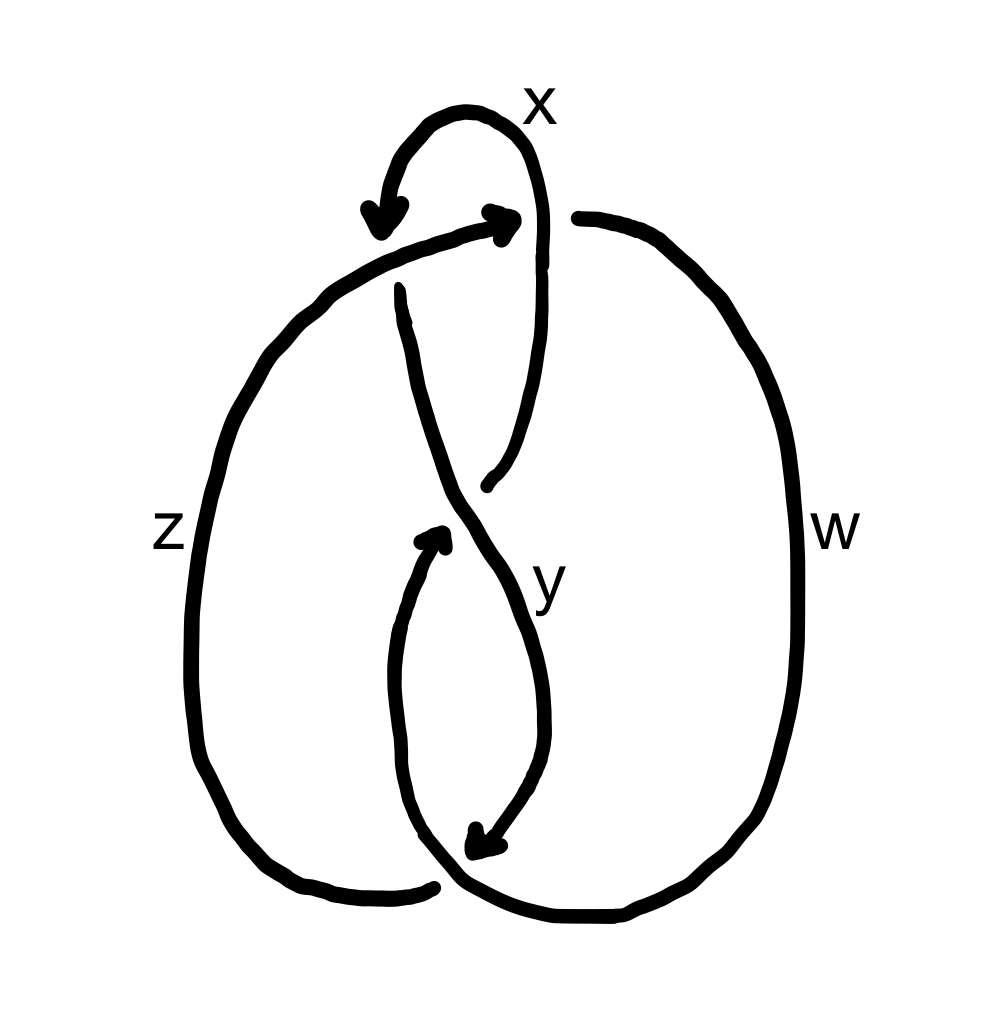}
    \caption{Oriented figure-8 knot with labeled arcs}
    \label{fig:figure8knot}
\end{figure}

\vspace{-.1cm}

%Given any finite quandle $X$, a coloring of a knot $K$ by $X$ is a quandle homomorphism from the fundamental quandle $Q(K)$ to the quandle $X$.  
The set of quandle homomorphisms from the fundamental quandle $Q(K)$ to the quandle $X$ is denoted by $\textup{Hom}(Q(K),X)$.  This set has been used to define computable link invariants, for example, the cardinality of the set $\textup{Hom}(Q(K),X)$ is known as the \emph{quandle counting invariant} \cite{EN}.

\textbf{Notations:}
Using a similar notation as in \cite{ZL}, we denote the directed complete graph on $n$ vertices with $\omega$ total number of edges between any two vertices in the graph as $(\overleftrightarrow{K_n}, \hat \omega)$. When two graphs $G_1$ and $G_2$ have disjoint vertex sets and edge sets, and there are $d$ edges from each vertex of $G_2$ to each vertex of $G_1$, the graph will be denoted as $G_1\overleftarrow{\nabla}_{\hat{d}}G_2$. Lastly, the disjoint union of $m$ copies of the graph $G$ will be denoted as $\bigsqcup\limits_{m}G$.

 \textbf{Known results:}
The only cases where quandle coloring quivers for torus links are completely known are the \emph{two} cases of $T(2,q)$ and $T(3,q)$.  In \cite{BC} the structure of the quandle coloring quiver of a $T(2,q)$ with respect to a dihedral quandle of order $n$, denoted $R_n$, was given.  In \cite{ZL} the coloring quiver was studied for $T(3,q)$. We give the known results for $\mathcal{Q}_{R_n} (T(p,q))$ for $p=2,3$ in the table below.
\\

\begin{table}[H]
  \centering
  \scalebox{1}{%
\begin{tabular}{|c| c| c|} 
 \hline
 Number of colorings $N$ & Quivers when $p=2$ & Quivers when $p=3$ \\ [0.5ex] 
 \hline\hline
 When $N=n$ &$(\overleftrightarrow{K_n}, \hat n)$ &$(\overleftrightarrow{K_n}, \hat n)$  \\
 \hline
  When $N=pn$ &$(\overleftrightarrow{K_n}, \hat n)\overleftarrow{\nabla}_{\hat{\frac{n}{2}}}\left(\overleftrightarrow{K_{n}}, \hat {\frac{n}{2}}\right)$  & $(\overleftrightarrow{K_n}, \hat n)\overleftarrow{\nabla}_{\hat{\frac{n}{3}}}\left(\overleftrightarrow{K_{2n}}, \hat {\frac{n}{3}}\right)$ \\ 
 \hline
  When $N=n\cdot2^{p-1}$&$(\overleftrightarrow{K_n}, \hat n)\overleftarrow{\nabla}_{\hat{\frac{n}{2}}}\left(\overleftrightarrow{K_{n}}, \hat {\frac{n}{2}}\right)$  &  $(\overleftrightarrow{K_n}, \hat n)\overleftarrow{\nabla}_{\hat{\frac{n}{2}}}\big[\bigsqcup \limits_{3}\left(\overleftrightarrow{K_{n}}, \hat {\frac{n}{2}}\right)\big]$    \\
 \hline
 When $N=n^p$& $(\overleftrightarrow{K_n}, \hat n)\overleftarrow{\nabla} \left(\overleftrightarrow{K_{n(n-1)}}, \hat {1}\right)$  &  $(\overleftrightarrow{K_n}, \hat n)\overleftarrow{\nabla} \big[\bigsqcup\limits_{n+1}\left(\overleftrightarrow{K_{n(n-1)}}, \hat {1}\right)\big]$   \\
 [1ex] 
 \hline
\end{tabular}
}
\caption{Quandle coloring quivers of $T(2,q)$ \cite{BC} and $T(3,q)$ \cite{ZL} by $R_n$.}
\end{table}
{\textbf{Our results:}  We solve the problem of determining the quandle coloring quivers of a general torus links $T(p,q)$ by dihedral quandles.  We give a summary of the result in the following chart.  For the details see the sections below. We begin by deriving the quandle coloring quivers in the $p=5$ and $p=7$ cases to give an overview of the process, then we generalize to $T(p,q)$ for any prime $p$.
}

\begin{table}[H]
  \centering
  \scalebox{1}{%
\begin{tabular}{|c| c| } 
 \hline
 Number of colorings $N$ & Quivers when $p$ is a prime \\ [0.5ex] 
 \hline\hline
 When $N=n$ &$(\overleftrightarrow{K_n}, \hat n)$  \\
 \hline
  When $N=pn$ & $(\overleftrightarrow{K_n}, \hat n)\overleftarrow{\nabla}_{\hat{\frac{n}{p}}}\left(\overleftrightarrow{K_{(p-1)n}}, \hat {\frac{n}{p}}\right)$ \\ 
 \hline
  When $N=n\cdot2^{p-1}$ &  $(\overleftrightarrow{K_n}, \hat n)\overleftarrow{\nabla}_{\hat{\frac{n}{2}}}\big[\bigsqcup\limits _{2^{p-1}-1}\left(\overleftrightarrow{K_{n}}, \hat {\frac{n}{2}}\right)\big]$    \\
 \hline
 When $N=n^p$ &  $(\overleftrightarrow{K_n}, \hat n)\overleftarrow{\nabla} \big[\bigsqcup\limits_{\frac{n^p-n}{n(n-1)}}\left(\overleftrightarrow{K_{n(n-1)}}, \hat {1}\right)\big]$   \\
 [1ex] 
 \hline
\end{tabular}
}
\caption{Quandle coloring quivers of $T(p,q)$ by $R_n$ (our main result).}
\end{table}

\section{Quandle Coloring Quivers of $(5,q)$ and $(7,q)$-Torus links with Dihedral Quandles}\label{Ex}
To further the understanding of deriving the quandle coloring quivers of $T(p,q)$-torus links with dihedral quandles, we go through the process of determining the number of colorings %from $Q(T(p,q))$ to $R_n$ 
of $T(p,q)$ by the dihedral quandle $R_n$ and the quivers when $p=5$ and when $p=7$.

\noindent
{\textbf{Quandle Coloring Quivers of $(5,q)$-Torus links with Dihedral Quandles:}}
\begin{comment}
Using the quandle presentation of fundamental quandle of the torus link $T(p,q)$ obtained in \cite{RSS}, $$Q(T(5,q))=\Big <x_1,x_2,...,x_5 \; \Big | 
\begin{aligned}
&x_i=x_{i+q} * x_q * x_{q-1}*\cdots * x_1 \text{ for } i \in \{1,...,5\};\\
&x_{5j+k}=x_k \text{ for } j\in \mathbb{Z}\text{, } k\in \{1,...,5\}
\end{aligned}
\Big >,$$
\end{comment}
Using the braid form $(\sigma_1 \sigma_2 \sigma_3\sigma_4)^q$ of the torus link $T(5,q)$ as shown in Figure~\ref{BraidT5q}, we calculate $N$, where $N$ is the number of colorings of $T(5,q)$ by the dihedral quandle $R_n$, for varying values of $q$. 

\begin{figure}[H]
    \centering
    \includegraphics[height=2.4in, width=2.8in]{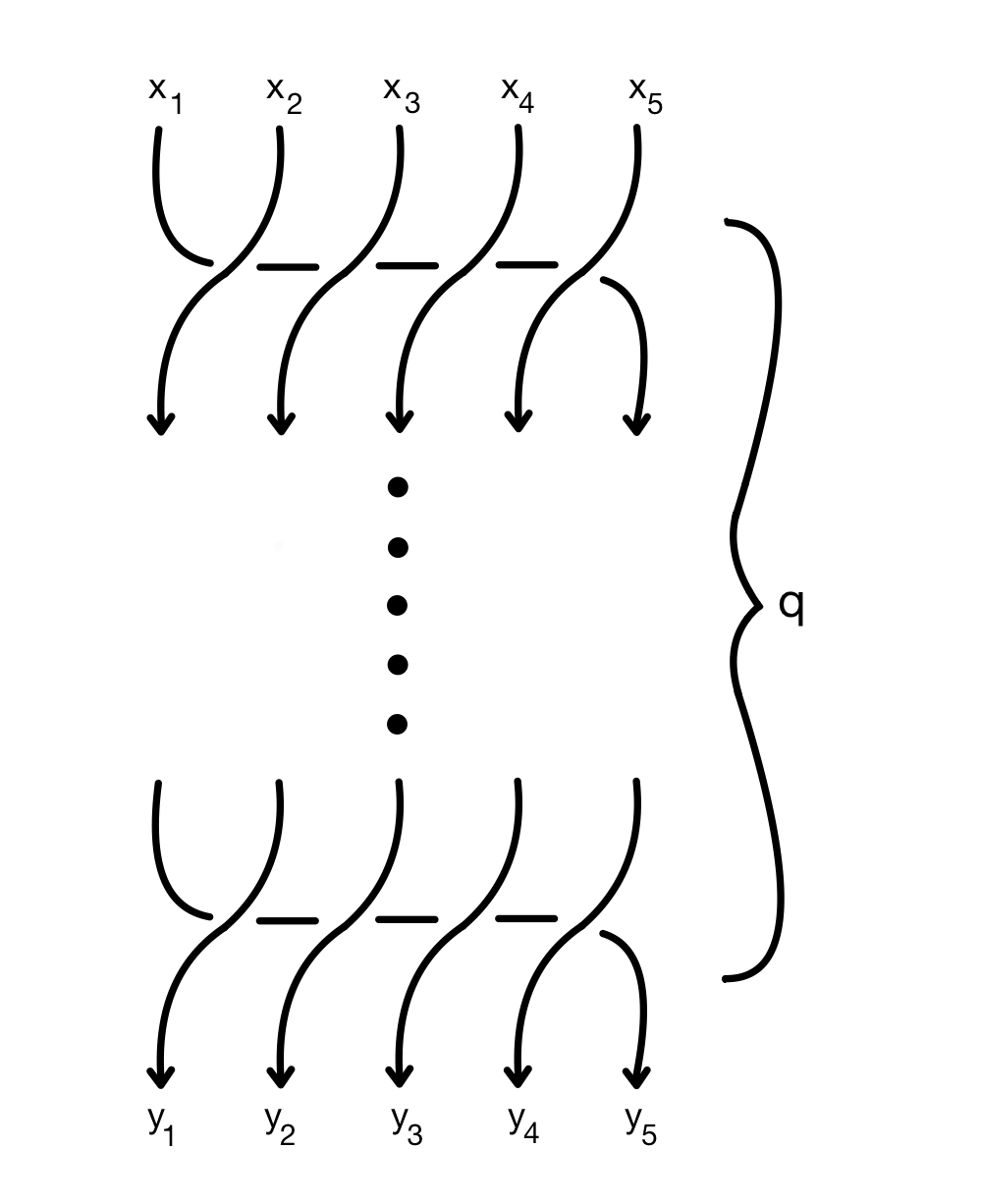}
    \caption{Coloring of the braid $(\sigma_1 \sigma_2 \sigma_3 \sigma_4)^q$ whose closure is $T(5,q)$}
    \label{BraidT5q}
\end{figure}

Let $f$ be a coloring of the torus link $T(5,q)$ %$f\in Hom(Q(T(5,q)),R_n)$ and 
with $f(x_i)=y_i$ for $i=1,...,5$ as in Figure~\ref{BraidT5q}. Coloring the braid $(\sigma_1 \sigma_2 \sigma_3 \sigma_4)^q$ by the dihedral quandle and using the fact that $f(x_i)=y_i$ gives the following equation:

\begin{comment}    
\begin{align*}
    f(x_5)&=y_5\\
    &=y_{5+q}* y_q* y_{q-1}*\cdots *y_1\\
    &=y_q* y_{q-1}*\cdots* y_1\\
    &=2y_1-2y_2+2y_3-2y_4+2y_5-2y_6+2y_7-2y_8+2y_9-2y_{10}\\
    &+\dots +(-1)^q2y_{q-1}+(-1)^{q+1}y_q \\
    &=2y_1-2y_2+2y_3-2y_4+2y_5-2y_1+2y_2-2y_3+2y_4-2y_{5}\\
    &+\dots +(-1)^q2y_{q-1}+(-1)^{q+1}y_q,
\end{align*}

\begin{align*}
   f(x_4)&=y_4\\
    &=y_{4+q}\cdot y_q\cdot y_{q-1}\cdots y_1\\
    &=2y_1-2y_2+2y_3-2y_4+2y_5-2y_6+2y_7-2y_8+2y_9-2y_{10}\\
    &+\dots +(-1)^q2y_{q-1}+(-1)^{q+1}2y_q+(-1)^{q+2}y_{4+q} \\
    &=y_5+(-1)^{q+1}y_q+(-1)^qy_{4+q},
\end{align*}
\begin{align*}
    f(x_3)&=y_3\\
    &=y_{3+q}\cdot y_q\cdot y_{q-1}\cdots y_1\\
    &=2y_1-2y_2+2y_3-2y_4+2y_5-2y_6+2y_7-2y_8+2y_9-2y_{10}\\
    &+\dots +(-1)^q2y_{q-1}+(-1)^{q+1}2y_q+(-1)^{q+2}y_{3+q} \\
    &=y_5+(-1)^{q+1}y_q+(-1)^qy_{3+q},
\end{align*}
\begin{align*}
   f(x_2)&=y_2\\
    &=y_{2+q}\cdot y_q\cdot y_{q-1}\cdots y_1\\
    &=2y_1-2y_2+2y_3-2y_4+2y_5-2y_6+2y_7-2y_8+2y_9-2y_{10}\\
    &+\dots +(-1)^q2y_{q-1}+(-1)^{q+1}2y_q+(-1)^{q+2}y_{2+q} \\
    &=y_5+(-1)^{q+1}y_q+(-1)^qy_{2+q},
\end{align*}\
and
\begin{align*}
    f(x_1)&=y_1\\
    &=y_{1+q}\cdot y_q\cdot y_{q-1}\cdots y_1\\
    &=2y_1-2y_2+2y_3-2y_4+2y_5-2y_6+2y_7-2y_8+2y_9-2y_{10}\\
    &+\dots +(-1)^q2y_{q-1}+(-1)^{q+1}2y_q+(-1)^{q+2}y_{1+q} \\
    &=y_5+(-1)^{q+1}y_q+(-1)^qy_{1+q}.
\end{align*}
\end{comment}

\begin{equation}\label{eq:1}
    y_i=y_5+(-1)^{q+1}y_q+(-1)^qy_{i+q}
\end{equation}
    
Now, using equation~(\ref{eq:1}), we calculate $N$ for the following values of $q$:

\begin{enumerate}
    \item Let $q=10k$. 
     A straightforward computation gives that all $y_i$ are free variables and thus the number of colorings is $N=n^5$.

   % \begin{align*}
   %     y_1&=y_5+(-1)^{10k+1}y_{10k}+(-1)^{10k}y_{10k+1}\\
    %    &=y_5-y_5+y_1\\
    %    &=y_1\\
     %   y_2&=y_5-y_5+y_2=y_2\\
      %  y_3&=y_3\\
       % y_4&=y_4\\
        % y_5&=y_5
   % \end{align*}

     %Since $y_i$ is free for all $i$, $N=n^5$.
     
    \item  Let $q=10k+1$. 
   % \begin{align*}
    %    y_1&=y_5+(-1)^{10k+2}y_{10k+1}+(-1)^{10k+1}y_{10k+2}\\
     %   &=y_5+y_1-y_2\\
      %  y_2&=y_5+y_1-y_3\\
       % y_3&=y_5+y_1-y_4\\
        %y_4&=y_5+y_1-y_5=y_1\\
        %y_5&=y_5+y_1-y_1=y_5
   % \end{align*}
The computation gives $y_1=y_2=y_3=y_4=y_5$. 
Therefore, $N=n$.

    \item Let $q=10k+2$. The computation gives $5(y_i-y_j)=0\mod n$. Therefore when $\gcd(5,n)=1$, we have $N=n$, and when $\gcd(5,n)=5$, we have $N=5n$.
    %\begin{align*}
     %   y_1&=y_5+(-1)^{10k+3}y_{10k+2}+(-1)^{10k+2}y_{10k+3}\\
      %  &=y_5-y_2+y_3\\
       % y_2&=y_5-y_2+y_4\\
        %y_3&=y_5-y_2+y_5=2y_5-y_2\\
        %y_4&=y_5-y_2+y_1\\
        %y_5&=y_5-y_2+y_2=y_5
    %\end{align*}
%$\implies$
   % \begin{align*}
    %    y_1&=3y_5-2y_2\\
     %   y_2&=5y_5-4y_2\\
      %  y_3&=y_5-y_2+y_5=2y_5-y_2\\
       % y_4&=4y_5-3y_2\\
        %y_5&=y_5
 %   \end{align*}
%$\implies 5(y_5-y_2)=0\mod n$. Therefore, we have two cases:
%\begin{enumerate}
%    \item When $\gcd(5,n)=1$, $N=n$.
%    \item When $\gcd(5,n)=5$, $N=5n$.
%\end{enumerate}  

    \item Let $q=10k+3$.
%    \begin{align*}
 %       y_1&=y_5+(-1)^{10k+4}y_{10k+3}+(-1)^{10k+3}y_{10k+4}\\
  %      &=y_5+y_3-y_4\\
   %     y_2&=y_5+y_3-y_5=y_3\\
    %    y_3&=y_5+y_3-y_1\\
     %   y_4&=y_5+y_3-y_2\\
      %  y_5&=y_5+y_3-y_3=y_5
   % \end{align*}
 The computation gives $y_1=y_2=y_3=y_4=y_5$. Therefore, $N=n$. 

    \item Let $q=10k+4$. The computation gives $5(y_i-y_j)=0\mod n$. Therefore when $\gcd(5,n)=1$, we have $N=n$, and when $\gcd(5,n)=5$, we have $N=5n$.
%    \begin{align*}
 %       y_1&=y_5+(-1)^{10k+5}y_{10k+4}+(-1)^{10k+4}y_{10k+5}\\
  %      &=y_5-y_4+y_5\\
   %     y_2&=y_5-y_4+y_1\\
    %    y_3&=y_5-y_4+y_2\\
     %   y_4&=y_5-y_4+y_3\\
      %  y_5&=y_5-y_4+y_4=y_5
    %\end{align*}
%$\implies$
 %   \begin{align*}
  %      y_1&=2y_5-y_4\\
   %     y_2&=3y_5-2y_4\\
    %    y_3&=4y_5-3y_4\\
     %   y_4&=5y_5-4y_4\\
    %    y_5&=y_5
    %\end{align*}
%$\implies 5(y_5-y_4)=0\mod n$. Therefore, we have two cases:
%\begin{enumerate}
 %   \item When $\gcd(5,n)=1$, $N=n$.
  %  \item When $\gcd(5,n)=5$, $N=5n$.
%\end{enumerate}  

    \item Let $q=10k+5$.
    The computation gives $2(y_i-y_j)=0\mod n$. Therefore when $\gcd(2,n)=1$, we have $N=n$, and when $\gcd(2,n)=2$, we have $N=2^4n$.
  %  \begin{align*}
   %     y_1&=y_5+(-1)^{10k+6}y_{10k+5}+(-1)^{10k+5}y_{10k+6}\\
    %    &=y_5+y_5-y_1\\
     %   y_2&=y_5+y_5-y_2\\
      %  y_3&=y_5+y_5-y_3\\
       % y_4&=y_5+y_5-y_4\\
       % y_5&=y_5+y_5-y_5=y_5
    %\end{align*}
%$\implies$
 %   \begin{align*}
  %      y_1&=2y_5-y_1\\
   %     y_2&=2y_5-y_2\\
    %    y_3&=2y_5-y_3\\
     %   y_4&=2y_5-y_4\\
      %  y_5&=y_5
    %\end{align*}
%$\implies 2(y_5-y_i)=0\mod n$.Therefore, we have two cases:
%\begin{enumerate}
 %   \item When $\gcd(2,n)=1$, $N=n$.
  %  \item When $\gcd(2,n)=2$, $N=2^4n$.
%\end{enumerate}  

    \item Let $q=10k+6$. The computation gives $5(y_i-y_j)=0\mod n$. Therefore when $\gcd(5,n)=1$, we have $N=n$, and when $\gcd(5,n)=5$, we have $N=5n$.
%    \begin{align*}
 %       y_1&=y_5+(-1)^{10k+7}y_{10k+6}+(-1)^{10k+6}y_{10k+7}\\
  %      &=y_5-y_1+y_2\\
   %     y_2&=y_5-y_1+y_3\\
    %    y_3&=y_5-y_1+y_4\\
     %   y_4&=y_5-y_1+y_5\\
      %  y_5&=y_5-y_1+y_1=y_5
    %\end{align*}
%$\implies$
 %   \begin{align*}
  %      y_1&=5y_5-4y_1\\
   %     y_2&=4y_5-3y_1\\
    %    y_3&=3y_5-2y_1\\
     %   y_4&=2y_5-y_1\\
      %  y_5&=y_5
    %\end{align*}
%$\implies 5(y_5-y_1)=0\mod n$. Therefore, we have two cases:
%\begin{enumerate}
 %   \item When $\gcd(5,n)=1$, $N=n$.
  %  \item When $\gcd(5,n)=5$, $N=5n$.
%\end{enumerate}  

    \item Let $q=10k+7$.
 %   \begin{align*}
  %      y_1&=y_5+(-1)^{10k+8}y_{10k+7}+(-1)^{10k+7}y_{10k+8}\\
   %     &=y_5+y_2-y_3\\
    %    y_2&=y_5+y_2-y_4\\
     %   y_3&=y_5+y_2-y_5=y_2\\
      %  y_4&=y_5+y_2-y_1\\
       % y_5&=y_5+y_2-y_2=y_5
    %\end{align*}
The computation gives $y_1=y_2=y_3=y_4=y_5$, and so $N=n$.

    \item Let $q=10k+8$.
    The computation gives $5(y_i-y_j)=0\mod n$. Therefore when $\gcd(5,n)=1$, we have $N=n$, and when $\gcd(5,n)=5$, we have $N=5n$.
%    \begin{align*}
 %       y_1&=y_5+(-1)^{10k+9}y_{10k+8}+(-1)^{10k+8}y_{10k+9}\\
  %      &=y_5-y_3+y_4\\
   %     y_2&=y_5-y_3+y_5\\
    %    y_3&=y_5-y_3+y_1\\
     %   y_4&=y_5-y_3+y_2\\
      %  y_5&=y_5-y_3+y_3=y_5
    %\end{align*}
%$\implies$
 %   \begin{align*}
  %      y_1&=4y_5-3y_3\\
   %     y_2&=2y_5-y_3\\
    %    y_3&=5y_5-4y_3\\
     %   y_4&=3y_5-2y_3\\
      %  y_5&=y_5
    %\end{align*}
%$\implies 5(y_5-5y_3)=0\mod n$. Therefore, we have two cases:
%\begin{enumerate}
 %   \item When $\gcd(5,n)=1$, $N=n$.
  %  \item When $\gcd(5,n)=5$, $N=5n$.
%\end{enumerate}  

    \item Let $q=10k+9$.
 %   \begin{align*}
  %      y_1&=y_5+(-1)^{10k+10}y_{10k+9}+(-1)^{10k+9}y_{10k+10}\\
   %     &=y_5+y_4-y_5=y_4\\
    %    y_2&=y_5+y_4-y_1\\
     %   y_3&=y_5+y_4-y_2\\
      %  y_4&=y_5+y_4-y_3\\
       % y_5&=y_5+y_4-y_4=y_5
    %\end{align*}
The computation gives $y_1=y_2=y_3=y_4=y_5$, and so $N=n$.
\end{enumerate}
We summarize these results in the following table below.
\begin{table}[h]
  \centering
  \scalebox{1}{%
\begin{tabular}{|c| c| } 
 \hline
 $q$ &
 Number of colorings $N$ when $p=5$\\ [0.5ex] 
 \hline\hline
  When $q\neq0\mod p$ is odd& $N=n$  \\
 [.5ex] 
 \hline
   When $q\neq 0\mod p$ is even& $N=n\cdot\gcd(5,n)$  \\
 \hline
 When $q=p\mod 2p$ & When $n$ is even, $N=2^4n$. When $n$ is odd, $N=n$   \\
 \hline
  When $q=0\mod 2p$ & $N=n^5$ \\ 
 \hline

\end{tabular}
}
\caption{Number of colorings of $T(5,q)$ by $R_n$}
\end{table}

%\paragraph{Conjecture}
%When $p=7$, we guess the following:
%\begin{itemize}
 %   \item When $q=1, 3, 5, 9, 11, 13$, $N=n$.
  %  \item When $q=2, 4, 6, 8, 10, 12$, $N=n$ or $7n$.
  %  \item When $q=0$, $N=n^7$.
  %  \item When $q=7$, $N=n$ or $2^6n$.
%\end{itemize}

We characterize in this section all the quandle coloring quivers of the torus links $T(5,q)$ with the dihedral quandle $R_n$.

In the following four theorems we let $T(5,q)$ be a torus link and $R_n $ be the dihedral quandle. 
\begin{thm}\label{Thm3.1}
        If $|Hom(T(5,q), R_n)|=n$, then the full quandle coloring quiver is the complete directed graph: \[
    \mathcal{Q}_{R_n} (T(5,q))=(\overleftrightarrow{K_n}, \hat n).
    \]
\end{thm}

\begin{proof}
    If $|Hom(T(5,q), R_n)|=n$, this corresponds to the $n$ trivial colorings.  Now assume that $f$ and $g$ are both trivial colorings given respectively by $f(x_i)=j$ and $g(x_i)=k$ for all $x_i \in Q(T(5,q))$.  Since any endomorphism of the dihedral quandle, $R_n$, has the form 
    \begin{equation}\label{eq:2}
        \phi(f)=af+b
    \end{equation}
    where $a,b \in R_n$ (see \cite{EMR}), then there are exactly $n$ solutions of endomorphisms $\phi$ making equation~(\ref{eq:2}) hold.  Thus we obtain the complete graph $ \mathcal{Q}_{R_n} (T(5,q))=(\overleftrightarrow{K_n}, \hat n)$.
\end{proof}

\begin{thm}\label{thm3.2}
   If $|Hom(T(5,q), R_n)|=5n$, then the full quandle coloring quiver is the directed graph: \[
    \mathcal{Q}_{R_n} (T(5,q))=(\overleftrightarrow{K_n}, \hat n)\overleftarrow{\nabla}_{\hat{d}}\left(\overleftrightarrow{K_{4n}}, \hat {d}\right)
    ,\]where $d=\frac{n}{5}$. 
\end{thm}

\begin{proof}
    Let $f, g \in Hom (Q(T(5,q)), R_n)$ be two vertices of the quiver.  Since any endomorphism of the dihedral quandle $R_n$ is given by equation~(\ref{eq:2}), we need to find the number of endomorphisms $\phi$ making equation~(\ref{eq:2}) hold. We know when there are $N=5n$ colorings, then $gcd(5,n)=5$. Now, since there are no edges from a trivial coloring to a non-trivial coloring,  we consider the following $3$ cases 
    \begin{enumerate}
        \item \textbf {Case 1.}
        It is clear that if $f$ and $g$ are trivial colorings, then there are $n$ possible solutions to equation~(\ref{eq:2}). Therefore, we have $n$ edges between trivial colorings

        \item \textbf{Case 2.}
        When $f$ is a nontrivial coloring, and $g$ is a trivial coloring, so $g(x_i)=k$ for all $x_i$, equation~(\ref{eq:2}) becomes $a(f(x_i)-f(x_j))=0\mod n$ for all $i,j$.Thus, we have  $\frac{n}{5}$ edges from nontrivial colorings to trivial colorings.
        
        \item \textbf {Case 3.}
        When $f$ and $g$ are both non-trivial colorings, then $f$ and $g$ are given by $g(x_1)=g(x_3)=g(x_5) \neq g(x_2)=g(x_4)$, and $5(g(x_1)-g(x_2))=0$ modulo $n$. Since the equation
$ (g(x_1)-g(x_2))-a(f(x_1)-f(x_2))=0$ in $R_n$ has  $\frac{n}{5}$ solutions for $a$, we get the result.

%For example, see https://math.stackexchange.com/questions/3078090/the-number-of-solutions-to-xk-a-over-g-a-finite-cyclic-group-is-gcdk

    \end{enumerate}
    
\end{proof}

\begin{thm}\label{thm3.3}
    If $|Hom(T(5,q), R_n)|=16n$, then the full quandle coloring quiver is the  directed graph: \[
    \mathcal{Q}_{R_n} (T(5,q))= (\overleftrightarrow{K_n}, \hat n)\overleftarrow{\nabla}_{\hat{d}}\big[\bigsqcup\limits _{15}\left(\overleftrightarrow{K_{n}}, \hat {d}\right)\big]
    ,\]where $d=\frac{n}{2}$.
     
\end{thm}

\begin{proof}
    In order to have $N=16n$ colorings, we know $n$ must be even and for any coloring $f$, $2(f(x_i)-f(x_j))=0$ for any $i,j$. The $16n$ colorings consist of $n$ trivial colorings and $15n$ nontrivial colorings. By Theorem \ref{Thm3.1}, the $n$ trivial colorings correspond to a subgraph $(\overleftrightarrow{K_n}, \hat n)$ of $\mathcal{Q}_{R_n} (T(5,q))$. Now, there are $15n$ nontrivial colorings, so let $f$ be a nontrivial coloring, thus we have for some $j$, $f(x_i)=f(x_1)+m_{i-1}$ for $1<i\leq 5$ where $m_{i-1}=0$ or $\frac{n}{2}$. Since $n$ is even and $f$ is nontrivial, there exists $i$ such that $m_{i-1}=\frac{n}{2}$. We know for some coloring $g\in Hom(Q(T(5,q)),R_n)$, $g(x_i)=g(x_1)+m'_{i-1}$ where $m'_{i-1}=0$ or $\frac{n}{2}$ . By satisfying equation~(\ref{eq:2}), we have $am_{i-1}=m'_{i-1}$ for $1<i\leq 5$. If any $m_{i-1}=0$, then $m'_{i-1}=0$. Thus, consider the $m_{i-1}$ such that $m_{i-1}=\frac{n}{2}$ and so we have $a\frac{n}{2}=0$ or $n/2$. If $a\frac{n}{2}=0$ then $a$ is even, i.e. $g$ must be trivial, and we have $n/2$ edges between nontrivial colorings to trivial colorings. If $a\frac{n}{2}=\frac{n}{2}$ then $a$ is odd, so $m_{i-1}=m'_{i-1}=\frac{n}{2}$, and we have $n/2$ edges between nontrivial colorings to nontrivial colorings. For each choice of $m_1,m_2,m_3,m_4$, there are $n$ colorings  which divides the nontrivial colorings into 15 disjoint $\left(\overleftrightarrow{K_{n}}, \hat {\frac{n}{2}}\right)$.  
      
\end{proof}

\begin{thm}\label{Thm5.4}
   If $|Hom(T(5,q), R_n)|=n^5$, where $n$ is prime, then the full quandle coloring quiver is the  directed graph: \[
    \mathcal{Q}_{R_n} (T(5,q))=(\overleftrightarrow{K_n}, \hat n)\overleftarrow{\nabla} \big[\bigsqcup\limits_{m}\left(\overleftrightarrow{K_{n(n-1)}}, \hat {1}\right)\big],
    \] where $m=\frac{n^5-n}{n(n-1)}$.
\end{thm}
\begin{proof}
In this case, $y_1,...,y_5$ are free elements of $R_n$, so there are $n^5$ colorings (vertices), and we need to determine the edges in this quiver. By Theorem \ref{Thm3.1} we will have the $n$ trivial colorings corresponding to the subgraph $(\overleftrightarrow{K_n}, \hat n)$. Now, there are $n^5-n$ nontrivial colorings. Let $f$ be a nontrivial coloring, thus we have $$f(x_i)=f(x_1)+m_{i-1}$$ 
where $1< i\leq 5$ and $0\leq m_{i-1}\leq n-1$. 
Since $f$ is nontrivial, there exists an index $j$ such that $m_j\neq 0$. Let $g$ be another coloring, thus we have $$g(x_i)=g(x_1)+m'_{i-1}$$
where $1<i\leq 5$ and $0\leq m'_{i-1}\leq n-1$. 
To satisfy equation~(\ref{eq:2}), we get for $1<i\leq 5$, $$am_{i-1}=m'_{i-1}$$
If any $m_{i-1}=0$, then $m'_{i-1}=0$. Thus, we need only consider the $i-1$ such that $m_{i-1}\neq 0$. Let $\rho(n)$ be the Euler function of $n$, which is prime by hypothesis, and so $m_{i-1}^{\rho(n)}=1$. Therefore, we have a unique solution (an edge) for $a$ if the following equation is satisfied for all $i-1$ such that $m_{i-1}\neq 0$ :\begin{equation}\label{eq:3}
     m'_1m^{\rho (n)-1}_1=m'_2m^{\rho (n)-1}_2=m'_3m^{\rho (n)-1}_3=...=m'_{4}m^{\rho (n)-1}_{4}\
\end{equation}
Fix $m_j$ for each $j$ such that $m_j\neq 0$, so $1\leq m_j\leq n-1$. Now, we consider the following $3$ cases 
    \begin{enumerate}
        \item \textbf {Case 1: if $m_j=m'_j$ for all $j$:}\\
        It is clear that equation~(\ref{eq:3}) is satisfied in this case. Therefore there is an edge between the $n$ vertices for each choice of $m_1,...,m_{4}$. This divides the nontrivial colorings into $\frac{n^5-n}{n}$ $(\overleftrightarrow{K_n}, \hat 1)$ subgraphs.
        
        \item \textbf {Case 2: if $m_j=m'_j$ for some $j$ and $m_k\neq m'_k$ for some $k$:}\\
        In this case, it is impossible to satisfy equation~(\ref{eq:3}), since $1=m'_jm_j^{\rho(n)-1}\neq m'_km_k^{\rho(n)-1}$. Therefore there are no edges between colorings in this case.

        \item \textbf {Case 3: if $m_j\neq m'_j$ for all $j$:}\\
        We have $n-2$ options for $m'_j$, which determines the other $m'_k$'s to satisfy equation~(\ref{eq:3}). Therefore there is an edge between each vertex of the $\frac{n^5-n}{n}$ $(\overleftrightarrow{K_n}, \hat 1)$ subgraphs from Case 1, to each vertex of $n-2$ other $\frac{n^5-n}{n}$ $(\overleftrightarrow{K_n}, \hat 1)$ subgraphs. This means we now have subgraphs $(\overleftrightarrow{K_{n(n-1)}}, \hat 1)$. We thus have $\frac{n^5-n}{n(n-1)}$ $(\overleftrightarrow{K_{n(n-1)}}, \hat 1)$ subgraphs. 

        Case 2 tells us that the $(\overleftrightarrow{K_{n(n-1)}}, \hat 1)$ subgraphs are disjoint. However, we have an edge from each nontrivial coloring to each trivial coloring since equation~(\ref{eq:3}) is always satisfied if $m'_{i-1}=0$ for all $i-1$. This gives the desired result.
        
    \end{enumerate}
\end{proof}

We give the following examples of the full quandle coloring quiver for varying values of $N$.\\

\begin{ex}
    When $N=n$, we have $\mathcal{Q}_{R_n} (T(5,q))=(\overleftrightarrow{K_n}, \hat n)$ by Theorem \ref{Thm3.1}. In Figure \ref{fig:1}, we show the graph $(\overleftrightarrow{K_5}, \hat 5)$ corresponding to the full $R_5$-quandle coloring quiver of $T(5,q)$, while Figure \ref{fig:2} shows $\mathcal{Q}_{R_{20}} (T(5,q))=(\overleftrightarrow{K_{20}}, \hat {20})$. For visual simplicity, we leave out arrows on the graph edges and remind the reader that the weight of each directed edge is $n$. \\
    \begin{figure}[h]
        \centering
        \scalebox{0.8}{
         \begin{tikzpicture}
            \graph { subgraph K_n [n=5,clockwise,radius=1.2cm, empty nodes, nodes={fill,circle,inner sep=2pt, minimum size=1pt}] };
        \end{tikzpicture}}
        \caption{$(\overleftrightarrow{K_5}, \hat 5)$}
        \label{fig:1}
    \end{figure}
    
    \begin{figure}[H]
    
        \centering
        \scalebox{0.85}{
         \begin{tikzpicture}

% Nodes for K_20
\foreach \i in {1,...,20}
    \node[draw, circle, fill, minimum size=1pt, inner sep=0pt] (A\i) at ({360/20*(\i-1)}:5cm) {};

% Edges for K_20
\foreach \i in {1,...,20}
    \foreach \j in {\i,...,20}
        \draw[opacity=0.3] (A\i) -- (A\j);

\end{tikzpicture}}
        \caption{$(\overleftrightarrow{K_{20}}, \hat {20})$}
        \label{fig:2}
    \end{figure}

\end{ex}

\begin{ex}
    By Theorem \ref{thm3.2}, when $N=5n$, we have \[
    \mathcal{Q}_{R_n} (T(5,q))=(\overleftrightarrow{K_n}, \hat n)\overleftarrow{\nabla}_{\hat{d}}\left(\overleftrightarrow{K_{4n}}, \hat {d}\right)
    ,\]where $d=\frac{n}{5}$. The full $R_5$-quandle coloring quiver of $T(5,q)$ is thus, \[
    \mathcal{Q}_{R_5} (T(5,q))=(\overleftrightarrow{K_5}, \hat 5)\overleftarrow{\nabla}_{\hat{1}}\left(\overleftrightarrow{K_{20}}, \hat {1}\right).\]  Notice that the subgraph components of the quiver $\mathcal{Q}_{R_5} (T(5,q))$ can be seen in Figures \ref{fig:1} and \ref{fig:2}. In Figure \ref{fig:3} we draw $\mathcal{Q}_{R_5} (T(5,q))$ as a simplified graph where the blue (left) and red (right) dots represent $(\overleftrightarrow{K_5}, \hat 5)$ and $(\overleftrightarrow{K_{20}}, \hat {20})$, respectively, and the directed edge indicates a directed edge with weight 1 from every vertex of $(\overleftrightarrow{K_{20}}, \hat {20})$ to every vertex of $(\overleftrightarrow{K_5}, \hat 5)$.

    \begin{figure}[h]
        \centering
        \begin{tikzpicture}

% Vertices
\node[draw, circle, fill=red, minimum size=2pt, inner sep=3pt, text=white] (Red) at (3,0) {};
\node[draw, circle, fill=blue, minimum size=2pt, inner sep=3pt, text=white] (Blue) at (0,0) {};

% Directed edge from Red to Blue
\draw[->, thick] (Red) -- (Blue);

\end{tikzpicture}
        \caption{$\mathcal{Q}_{R_5} (T(5,q))=(\overleftrightarrow{K_5}, \hat 5)\overleftarrow{\nabla}_{\hat{1}}\left(\overleftrightarrow{K_{20}}, \hat {1}\right)$}
        \label{fig:3}
    \end{figure}
\end{ex}

\begin{ex}
    By Theorem \ref{thm3.3}, when $N=16n$, we have \[
    \mathcal{Q}_{R_n} (T(5,q))= (\overleftrightarrow{K_n}, \hat n)\overleftarrow{\nabla}_{\hat{d}}\big[\bigsqcup\limits _{15}\left(\overleftrightarrow{K_{n}}, \hat {d}\right)\big]
    ,\]where $d=\frac{n}{2}$. Thus, if $n=6$, we know $
    \mathcal{Q}_{R_6} (T(5,q))= (\overleftrightarrow{K_6}, \hat 6)\overleftarrow{\nabla}_{\hat{3}}\big[\bigsqcup\limits _{15}\left(\overleftrightarrow{K_{6}}, \hat {3}\right)\big]
    $. In Figure \ref{fig:4} we again use the simplified technique of illustrating this quiver, where the blue dot (center) represents $(\overleftrightarrow{K_6}, \hat 6)$ and each of the 15 green dots represent $(\overleftrightarrow{K_6}, \hat 3)$. Each of the directed edges from a green dot to a blue dot indicates that there is a directed edge from each vertex of $(\overleftrightarrow{K_6}, \hat 3)$ to every vertex of $(\overleftrightarrow{K_6}, \hat 6)$.

    \begin{figure}[h]
        \centering
        \scalebox{0.7}{
        \begin{tikzpicture}

% Blue central vertex
\node[draw, circle, fill=blue, minimum size=1pt, inner sep=2pt, text=white] (Blue) at (0,0) {} ;

% 15 green vertices and directed edges pointing to the blue vertex
\foreach \i in {1,...,15} {
    \node[draw, circle, fill=green, minimum size=1pt, inner sep=2pt, text=white] (Green\i) at ({360/15*(\i-1)}:4cm) {};
    \draw[-{Stealth[length=7pt, width=3pt]}, line width=1pt] (Green\i) -- (Blue);
}

\end{tikzpicture}}
        \caption{$
    \mathcal{Q}_{R_6} (T(5,q))= (\overleftrightarrow{K_6}, \hat 6)\overleftarrow{\nabla}_{\hat{3}}\big[\bigsqcup\limits _{15}\left(\overleftrightarrow{K_{6}}, \hat {3}\right)\big]
    $}
        \label{fig:4}
    \end{figure}
\end{ex}

%\subsection*{Quandle Coloring Quivers of $(7,q)$-Torus links with Dihedral Quandles}

\textbf{Quandle Coloring Quivers of $(7,q)$-Torus links with Dihedral Quandles}
We calculate the quiver of $T(7,q)$ when $q=14k,14k+1,14k+2,\ldots,14k+13$.
Let $N$ be the number of colorings of $T(7,q)$ by $R_n$.

Using the braid form $(\sigma_1 \sigma_2 \ldots \sigma_6)^q$ of $T(7,q)$, we calculate $N$ for varying values of $q$. Let $f$ be a coloring of $T(7,q)$ by $R_n$, and $f(x_i)=y_i$ for $i=1,...,7$. Coloring the braid $(\sigma_1 \sigma_2 \sigma_3 \sigma_4 \sigma_5 \sigma_6)^q$ by the dihedral quandle and using the fact that $f(x_i)=y_i$ gives the following equation:
%\begin{align*}
 %   f(x_7)&=y_7\\
  %  &=y_{7+q}* y_q* y_{q-1}*\cdots * y_1\\
   % &=y_q\cdot y_{q-1}\cdots y_1\\
    %&=2y_1-2y_2+2y_3-2y_4+2y_5-2y_6+2y_7\\
    %&-2y_8+2y_9-2y_{10}+2y_{11}-2y_{12}+2y_{13}-2y_{14}\\
    %&+\dots +(-1)^q2y_{q-1}+(-1)^{q+1}y_q \\
    %&=2y_1-2y_2+2y_3-2y_4+2y_5-2y_6+2y_7\\
    %&-2y_1+2y_2-2y_3+2y_4-2y_{5}+2y_6-2y_7\\
    %&+\dots +(-1)^q2y_{q-1}+(-1)^{q+1}y_q
%\end{align*}
%\begin{align*}
 %   f(x_6)&=y_6\\
  %  &=y_{6+q}\cdot y_q\cdot y_{q-1}\cdots y_1\\
   % &=2y_1-2y_2+2y_3-2y_4+2y_5-2y_6+2y_7\\
    %&-2y_8+2y_9-2y_{10}+2y_{11}-2y_{12}+2y_{13}-2y_{14}\\
    %&+\dots +(-1)^q2y_{q-1}+(-1)^{q+1}2y_q+(-1)^qy_{6+q} \\
    %&=y_7+(-1)^{q+1}y_q+(-1)^qy_{6+q}
%\end{align*}
%\begin{align*}
 %   f(x_5)&=y_5\\
  %  &=y_{5+q}\cdot y_q\cdot y_{q-1}\cdots y_1\\
   % &=y_7+(-1)^{q+1}y_q+(-1)^qy_{5+q}
%\end{align*}
%\begin{align*}
 %   f(x_4)&=y_4\\
  %  &=y_{4+q}\cdot y_q\cdot y_{q-1}\cdots y_1\\
   % &=y_7+(-1)^{q+1}y_q+(-1)^qy_{4+q}
%\end{align*}
%\begin{align*}
 %   f(x_3)&=y_3\\
  %  &=y_{3+q}\cdot y_q\cdot y_{q-1}\cdots y_1\\
   % &=y_7+(-1)^{q+1}y_q+(-1)^qy_{3+q}
%\end{align*}
%\begin{align*}
 %   f(x_2)&=y_2\\
  %  &=y_{2+q}\cdot y_q\cdot y_{q-1}\cdots y_1\\
   % &=y_7+(-1)^{q+1}y_q+(-1)^qy_{2+q}
%\end{align*}
%\begin{align*}
 %   f(x_1)&=y_1\\
  %  &=y_{1+q}\cdot y_q\cdot y_{q-1}\cdots y_1\\
   % &=y_7+(-1)^{q+1}y_q+(-1)^qy_{1+q}
%\end{align*} 
\begin{equation}\label{eq:4}
    y_i=y_7+(-1)^{q+1}y_q+(-1)^qy_{i+q}
\end{equation}

Now, using equation~(\ref{eq:4}) and following a similar process as in the $p=5$ case, we summarize the results for $p=7$ in the following table below:

\begin{table}[h]
  \centering
  \scalebox{1}{%
\begin{tabular}{|c| c| } 
 \hline
 $q$ & Number of colorings $N$ when $p=7$\\ [0.5ex] 
 \hline\hline
 When q=p & When $n=7k$ for $k\geq 2$, $N=2^6n$   \\
 [1ex]
 \hline
  When q=p & When gcd(7,n)=1 or n=7,  $N=n$ \\ 
  [1ex]
 \hline
  When q is even& $N=n\cdot\gcd(7,n)$  \\
  [1ex]
 \hline
 When q is odd, not p& $N=n$  \\
 [1ex] 
 \hline
\end{tabular}
}
\caption{Number of colorings of $T(7,q)$ by $R_n$}
\end{table}

%\newpage

\begin{thm}\label{Thm3.5}
        If $|Hom(T(7,q), R_n)|=n$, then the full quandle coloring quiver is the complete directed graph: \[
    \mathcal{Q}_{R_n} (T(7,q))=(\overleftrightarrow{K_n}, \hat n).
    \]
\end{thm}
\begin{proof}
    If $|Hom(T(7,q), R_n)|=n$ (this corresponds to trivial colorings).  Now assume that $f$ and $g$ are both trivial colorings given respectively by $f(x_i)=j\; \textit{and}\; g(x_i)=k$ for all $x_i \in Q(T(7,q))$.  Since any endomorphism of the dihedral quandle $R_n$ has the form $\phi(f)=af+b$, where $a,b \in R_n$, then there are exactly $n$ solutions of functions $\phi$ making the equation $\phi(f)=g$ hold.  Thus we obtain the complete graph $ \mathcal{Q}_{R_n} (T(7,q))=(\overleftrightarrow{K_n}, \hat n)$.
\end{proof}

\begin{thm}\label{Thm3.6}
   If $|Hom(T(7,q), R_n)|=7n$, then the full quandle coloring quiver is the directed graph: \[
    \mathcal{Q}_{R_n} (T(7,q))=(\overleftrightarrow{K_n}, \hat n)\overleftarrow{\nabla}_{\hat{d}}\left(\overleftrightarrow{K_{6n}}, \hat {d}\right)
    ,\]where $d=\frac{n}{7}$. 
\end{thm}
\begin{proof}
     Let $f,g \in Hom (Q(T(7,q)), R_n)$ be two vertices of the quiver.  Since any endomorphism of the dihedral quandle $R_n$ is given by $\phi(f)=af+b$ (see \cite{EMR}), we need to find the number of endomorphisms $\phi$ making the equation~(\ref{eq:2}) hold. We know when there are $N=7n$ colorings, then $gcd(7,n)=7$. Now, since there are no edges from a trivial coloring to a non-trivial coloring,  we consider the following $3$ cases:
    \begin{enumerate}
        \item \textbf {Case 1.}
        It is clear that if $f$ and $g$ are trivial colorings, i.e. constant maps, then there are $n$ possible solutions to equation~(\ref{eq:2}), and we have $n$ edges between trivial colorings.
        
        \item \textbf {Case 2.}
        When $f$ is a non-trivial coloring, then for some trivial coloring $g$ we have $g(x_i)=k$ for all $x_i$, and equation~(\ref{eq:2}) becomes $a(f(x_i)-f(x_j))=0 \mod n$ for all $i,j$. Since $7(f(x_i)-f(x_j))=0\mod n$, we have $\frac{n}{7}$ from nontrivial colorings to trivial colorings.

        \item \textbf {Case 3.}
        When $f$ and $g$ are both non-trivial colorings, then $f$ and $g$ are given by $7(f(x_i)-f(x_j))=7(g(x_i)-g(x_j))=0$ modulo $n$. Since the equation
$ (g(x_i)-g(x_j))-a(f(x_i)-f(x_j))=0$ in $R_n$ has  $\frac{n}{7}$ solutions for $a$, we get the result.
        \

    \end{enumerate}
\end{proof}

\begin{thm}\label{Thm3.7}
    If $|Hom(T(7,q), R_n)|=64n$, then the full quandle coloring quiver is the  directed graph: 
    \[
    \mathcal{Q}_{R_n} (T(7,q))= (\overleftrightarrow{K_n}, \hat n)\overleftarrow{\nabla}_{\hat{d}}\big[\bigsqcup\limits _{63}\left(\overleftrightarrow{K_{n}}, \hat {d}\right)\big]
    ,\]where $d=\frac{n}{2}$.
     
\end{thm}
\begin{proof}
    In order to have $N=64n$ colorings, we know $n$ must be even and for any coloring $f$, $2(f(x_i)-f(x_j))=0$ for any $i,j$. The $64n$ colorings consist of $n$ trivial colorings and $63n$ nontrivial colorings. By Theorem \ref{Thm3.5}, the $n$ trivial colorings correspond to a subgraph $(\overleftrightarrow{K_n}, \hat n)$ of $\mathcal{Q}_{R_n} (T(7,q))$. Now, there are $63n$ nontrivial colorings, so let $f$ be a nontrivial coloring, thus we have, $f(x_i)=f(x_1)+m_{i-1}$ for $1<i\leq 7$ where $m_{i-1}=0$ or $\frac{n}{2}$. Since $n$ is even and $f$ nontrivial, there exists $i$ such that $m_{i-1}=\frac{n}{2}$. We know for some coloring $g\in Hom(Q(T(7,q)),R_n)$, $g(x_i)=g(x_1)+m'_{i-1}$ where $m'_{i-1}=0$ or $\frac{n}{2}$ . By satisfying equation~(\ref{eq:2}), we have $am_{i-1}=m'_{i-1}$ for $1<i\leq 7$. If any $m_{i-1}=0$, then $m'_{i-1}=0$. Thus, consider the $m_{i-1}$ such that $m_{i-1}=\frac{n}{2}$ and so we have $a\frac{n}{2}=0$ or $n/2$. If $a\frac{n}{2}=0$ then $a$ is even, i.e. $g$ must be trivial, and we have $n/2$ edges between nontrivial colorings to trivial colorings. If $a\frac{n}{2}=\frac{n}{2}$ then $a$ is odd ($m_{i-1}=m'_{i-1}=\frac{n}{2}$ and we have $n/2$ edges between nontrivial colorings to nontrivial colorings. For each choice of $m_1,m_2,...,m_6$, there are $n$ colorings  which divides the nontrivial colorings into 63 disjoint $\left(\overleftrightarrow{K_{n}}, \hat {\frac{n}{2}}\right)$.    
\end{proof}

\begin{thm}\label{Thm3.8}
   If $|Hom(T(7,q), R_n)|=n^7$, then the full quandle coloring quiver is the  directed graph: \[
    \mathcal{Q}_{R_n} (T(7,q))=(\overleftrightarrow{K_n}, \hat n)\overleftarrow{\nabla} \big[\bigsqcup\limits_{m}\left(\overleftrightarrow{K_{n(n-1)}}, \hat {1}\right)\big],
    \] where $m=\frac{n^7-n}{n(n-1)}$.
\end{thm}
\begin{proof}
We know the quiver will have a subgraph $(\overleftrightarrow{K_n}, \hat n)$ corresponding to the $n$ trivial colorings. Let $f$ be a nontrivial coloring, thus we have $$f(x_i)=f(x_1)+m_{i-1}$$ 
where $1< i\leq 7$ and $0\leq m_{i-1}\leq n-1$. 
Since $f$ is nontrivial, there exists an index $j$ such that $m_j\neq 0$. Let $g$ be another coloring, thus we have $$g(x_i)=g(x_1)+m'_{i-1}$$
where $1<i\leq 7$ and $0\leq m'_{i-1}\leq n-1$. 
To satisfy equation~(\ref{eq:2}), we get for $1<i\leq 7$, $$am_{i-1}=m'_{i-1}$$
If any $m_{i-1}=0$, then $m'_{i-1}=0$. Thus, we need only consider the $i-1$ such that $m_{i-1}\neq 0$. Let $\rho(n)$ be the Euler function of $n$, which is prime by hypothesis, and so $m_{i-1}^{\rho(n)}=1$. Therefore, we have a unique solution (an edge) for $a$ if the following equation is satisfied for all $i-1$ such that $m_{i-1}\neq 0$:\begin{equation}\label{eq:5}
     m'_1m^{\rho (n)-1}_1=m'_2m^{\rho (n)-1}_2=m'_3m^{\rho (n)-1}_3=...=m'_{6}m^{\rho (n)-1}_{6}\
\end{equation}
Fix $m_i$ for each $i$ such that $m_i\neq 0$, so $1\leq m_i\leq n-1$. Now, we consider the following $3$ cases 
    \begin{enumerate}
        \item \textbf {Case 1: if $m_i=m'_i$ for all $i$:}\\
        It is clear that equation~(\ref{eq:5}) is satisfied in this case. Therefore there is an edge between the $n$ vertices for each choice of $m_1,...,m_6$. This divides the nontrivial colorings into $\frac{n^7-n}{n}$ $(\overleftrightarrow{K_n}, \hat 1)$ subgraphs.
        
        \item \textbf {Case 2: if $m_i=m'_i$ for some $i$ and $m_j\neq m'_j$ for some $j$:}\\
        In this case, it is impossible to satisfy equation~(\ref{eq:5}), since $m'_jm_j^{\rho(n)-1}\neq 1=m'_im_i^{\rho(n)-1}$. Therefore there are no edges between colorings in this case.

        \item \textbf {Case 3: if $m_i\neq m'_i$ for all $i$:}\\
        We have $n-2$ options for $m'_i$, which determines the other $m'_j$'s to satisfy equation~(\ref{eq:5}). Therefore there is an edge between each vertex of the $\frac{n^7-n}{n}$ $(\overleftrightarrow{K_n}, \hat 1)$ subgraphs from Case 1, to each vertex of $n-2$ other $\frac{n^7-n}{n}$ $(\overleftrightarrow{K_n}, \hat 1)$ subgraphs. This means we now have subgraphs $(\overleftrightarrow{K_{n(n-1)}}, \hat 1)$. We thus have $\frac{n^7-n}{n(n-1)}$ $(\overleftrightarrow{K_{n(n-1)}}, \hat 1)$ subgraphs. 

        Case 2 tells us that the $(\overleftrightarrow{K_{n(n-1)}}, \hat 1)$ subgraphs are disjoint. However, we have an edge from each nontrivial coloring to each trivial coloring since equation~(\ref{eq:5}) is always satisfied if $m'_i=0$ for all $i$. This gives the desired result.
        
    \end{enumerate}
\end{proof}

\section{Quandle Coloring Quivers of General $(p,q)$-Torus links with Dihedral Quandles}\label{Gen}
\begin{comment}
    
Using the %quandle 
presentation of the fundamental quandle of the torus link $T(p,q)$, 
$$Q(T(p,q))=\Big <x_1,x_2,...,x_p\Big | 
\begin{aligned}
&x_i=x_{i+q} * x_q * x_{q-1}*\cdots * x_1 \text{ for } i \in \{1,...,p\};\\
&x_{pj+k}=x_k \text{ for } j\in \mathbb{Z}\text{, } k\in \{1,...,p\}
\end{aligned}
\Big >,$$
\end{comment}

Using the braid form $(\sigma_1\sigma_2...\sigma_{p-1})^q$ of $T(p,q)$, we calculate the number of colorings, $N$, for varying values of $q$ when $p$ is an odd prime. Let $f$ be a coloring of $T(p,q)$ by $R_n$ and $f(x_i)=y_i$ for $i=1,...,p$. Coloring the braid $(\sigma_1 \sigma_2 ...  \sigma_{p-1})^q$ by the dihedral quandle and using the fact that $f(x_i)=y_i$ gives the following equation:
\begin{equation}\label{eq:6}
    y_i=y_p+(-1)^{q+1}y_q+(-1)^qy_{i+q}
\end{equation}
Thus, we have the following:
\begin{itemize}
    \item When $q=2pk$, $N=n^p$.
    \item When $q=2pk+i$ where $i<2p$ is an odd number but $i\neq p$, $N=n$.
    \item When $q=2pk+j$ where $j<2p$ is an even number, \begin{itemize}
        \item When $n=p$ or $\gcd(n,p)=1$, $N=n$;
        \item When $n\geq 2p$ and $\gcd(n,p)=p$, $N=pn$.
    \end{itemize}
    \item When $q=2pk+p$, 
    \begin{itemize}
    \item When $\gcd(2,n)=1$, $N=n$.
    \item When $\gcd(2,n)=2$, $N=2^{p-1}n$.
    \end{itemize}
\end{itemize}

\begin{thm}\label{Thm4.1}
        If $|Hom(T(p,q), R_n)|=n$, then the full quandle coloring quiver is the complete directed graph: \[
    \mathcal{Q}_{R_n} (T(p,q))=(\overleftrightarrow{K_n}, \hat n).
    \]
\end{thm}
\begin{proof}
    We know if $|Hom(T(p,q), R_n)|=n$, this corresponds to the $n$ trivial colorings. Now assume that $f$ and $g$ are both trivial colorings given respectively by $f(x_i)=j\; \textit{and}\; g(x_i)=k$ for all $x_i \in Q(T(p,q))$.  Since any endomorphism of the dihedral quandle $R_n$ is given by $\phi(f)=af+b$ (see \cite{EMR}), we need to find the number of endomorphisms $\phi$ making the equation 
    \begin{eqnarray}\label{eq:7}
        \phi(f)=af+b=g
    \end{eqnarray} hold, where $a,b \in R_n$. there are exactly $n$ solutions of functions $\phi$ making the equation $\phi(f)=g$ hold.  Thus we obtain the complete graph $ \mathcal{Q}_{R_n} (T(p,q))=(\overleftrightarrow{K_n}, \hat n)$.
\end{proof}

\begin{thm}\label{thm4.2}
   If $|Hom(T(p,q), R_n)|=pn$, then the full quandle coloring quiver is the directed graph: \[
    \mathcal{Q}_{R_n} (T(p,q))=(\overleftrightarrow{K_n}, \hat n)\overleftarrow{\nabla}_{\hat{d}}\left(\overleftrightarrow{K_{(p-1)n}}, \hat {d}\right)
    ,\]where $d=\frac{n}{p}$. 
\end{thm}

\begin{proof}
     Let $f,g \in Hom (Q(T(p,q)), R_n)$ be two vertices of the quiver.  Since any endomorphism of the dihedral quandle $R_n$ is given by $\phi(f)=af+b$, we need to find the number of endomorphisms $\phi$ making the equation~(\ref{eq:7}) hold. We know when there are $N=pn$ colorings, then $gcd(p,n)=p$. Now, since there are no edges from a trivial coloring to a non-trivial coloring,  we consider the following $3$ cases:
    \begin{enumerate}
        \item \textbf {Case 1.}
        It is clear that if $f$ and $g$ are trivial colorings, i.e. constant maps, then there are $n$ possible solutions to equation~(\ref{eq:7}), and we have $n$ edges between trivial colorings.
        
        \item \textbf {Case 2.}
        When $f$ is a non-trivial coloring, then for some trivial coloring $g$ we have $g(x_i)=k$ for all $x_i$, and equation~(\ref{eq:7}) becomes $a(f(x_i)-f(x_j))=0 \mod n$ for all $i,j$. Since $p(f(x_i)-f(x_j))=0\mod n$, we have $\frac{n}{p}$ edges from nontrivial colorings to trivial colorings.

        \item \textbf {Case 3.}
        When $f$ and $g$ are both non-trivial colorings, then $f$ and $g$ are given by $p(f(x_i)-f(x_j))=p(g(x_i)-g(x_j))=0$ modulo $n$. Since the equation
$ (g(x_i)-g(x_j))-a(f(x_i)-f(x_j))=0$ in $R_n$ has  $\frac{n}{p}$ solutions for $a$, we get the result.
        \

    \end{enumerate}
\end{proof}

\begin{thm}\label{thm4.3}
    If $n$ is even and $|Hom(T(p,q), R_n)|=2^{p-1}n$, then the full quandle coloring quiver is the  directed graph: \[
    \mathcal{Q}_{R_n} (T(p,q))= (\overleftrightarrow{K_n}, \hat n)\overleftarrow{\nabla}_{\hat{d}}\big[\bigsqcup\limits _{2^{p-1}-1}\left(\overleftrightarrow{K_{n}}, \hat {d}\right)\big]
    ,\]where $d=\frac{n}{2}$.
     
\end{thm}

\begin{proof}
 When determining the number of colorings, $N$, we get the equation $2(f(x_i)-f(x_j))=0$ for any coloring $f$, and any $i,j.$ Since $f(x_i)$ takes $n$ possible values, and for $j \neq i$, $f(x_j)$ takes 2 values, thus giving  $N=n \cdot 2^{p-1}$.  This includes $n$ trivial colorings and $(2^{p-1}-1)n$ nontrivial colorings. By Theorem \ref{Thm4.1}, 
 the $n$ trivial colorings correspond to a subgraph $(\overleftrightarrow{K_n}, \hat n)$ of $\mathcal{Q}_{R_n} (T(p,q))$. Now, there are $(2^{p-1}-1)n$ nontrivial colorings, so let $f$ be a nontrivial coloring, thus we have, $f(x_i)=f(x_1)+m_{i-1}$ for $1<i\leq p$ where $m_{i-1}=0$ or $\frac{n}{2}$. Since $n$ is even and $f$ nontrivial, there exists $i$ such that $m_{i-1}=\frac{n}{2}$. We know for some coloring $g$ of $T(p,q)$ by $R_n$, we have $g(x_i)=g(x_1)+m'_{i-1}$ where $m'_{i-1}=0$ or $\frac{n}{2}$ . By satisfying equation~(\ref{eq:7}), we have $am_{i-1}=m'_{i-1}$ for $1<i\leq p$. If any $m_{i-1}=0$, then $m'_{i-1}=0$. Thus, consider the $m_{i-1}$ such that $m_{i-1}=\frac{n}{2}$ and so we have $a\frac{n}{2}=0$ or $n/2$. If $a\frac{n}{2}=0$ then $a$ is even, i.e. $g$ must be trivial, and we have $n/2$ edges between nontrivial colorings to trivial colorings. If $a\frac{n}{2}=\frac{n}{2}$ then $a$ is odd ($m_{i-1}=m'_{i-1}=\frac{n}{2}$ and we have $n/2$ edges between nontrivial colorings to nontrivial colorings. For each choice of $m_1,m_2,...,m_{p-1}$, there are $n$ colorings  which divides the nontrivial colorings into $2^{p-1}-1$ disjoint $\left(\overleftrightarrow{K_{n}}, \hat {\frac{n}{2}}\right)$.  
\end{proof}

\begin{thm}\label{Thm4.4}
   If $n$ is prime and $|Hom(T(p,q), R_n)|=n^p$, then the full quandle coloring quiver is the  directed graph: \[
    \mathcal{Q}_{R_n} (T(p,q))=(\overleftrightarrow{K_n}, \hat n)\overleftarrow{\nabla} \big[\bigsqcup\limits_{m}\left(\overleftrightarrow{K_{n(n-1)}}, \hat {1}\right)\big],
    \] where $m=\frac{n^p-n}{n(n-1)}$. 
     
\end{thm}

\begin{proof}
We know the quiver will have a subgraph $(\overleftrightarrow{K_n}, \hat n)$ corresponding to the $n$ trivial colorings. Let $f$ be a nontrivial coloring, so 
    $$f(x_i)=f(x_1)+m_{i-1}$$
for some $0\leq m_1,m_2,...,m_{p-1}\leq n-1$. Now, let $g\neq f$ be another coloring, so 
    $$g(x_i)=g(x_1)+m'_{i-1}$$
for some $1\leq m'_1,m'_2,...,m'_{p-1}\leq n$.
To satisfy equation~(\ref{eq:7}) , we get 
    $$a m_{i-1}=m'_{i-1}.$$
If any $m_i=0$, then $m'_i=0$. Thus, we need only consider the $i$ such that $m_i\neq 0$. Let $\rho(n)$ be the Euler function of $n$, which is prime by hypothesis, and so $m_{i-1}^{\rho(n)}=1$. Therefore, we have a unique solution (an edge) if the following equation is satisfied for all $i$ such that $m_i\neq 0$:\begin{equation}\label{eq:6}
     m'_1m^{\rho (n)-1}_1=m'_2m^{\rho (n)-1}_2=m'_3m^{\rho (n)-1}_3=...=m'_{p-1}m^{\rho (n)-1}_{p-1}\
\end{equation}
Fix $m_i$ for each $i$ such that $m_i\neq 0$, so $1\leq m_i\leq n-1$. Now, we consider the following $3$ cases 
    \begin{enumerate}
        \item \textbf {Case 1: if $m_i=m'_i$ for all $i$:}\\
        It is clear that equation~(\ref{eq:6}) is satisfied in this case. Therefore there is an edge between the $n$ vertices for each choice of $m_1,...,m_{p-1}$. This divides the nontrivial colorings into $\frac{n^p-n}{n}$ $(\overleftrightarrow{K_n}, \hat 1)$ subgraphs.
        
        \item \textbf {Case 2: if $m_i=m'_i$ for some $i$ and $m_j\neq m'_j$ for some $j$:}\\
        In this case, it is impossible to satisfy equation~(\ref{eq:6}), since $m'_jm_j^{\rho(n)-1}\neq 1=m'_im_i^{\rho(n)-1}$. Therefore there are no edges between colorings in this case.

        \item \textbf {Case 3: if $m_i\neq m'_i$ for all $i$:}\\
        We have $n-2$ options for $m'_i$, which determines the other $m'_j$'s to satisfy equation~(\ref{eq:6}). Therefore there is an edge between each vertice of the $\frac{n^p-n}{n}$ $(\overleftrightarrow{K_n}, \hat 1)$ subgraphs from Case 1, to each vertice of $n-2$ other $\frac{n^p-n}{n}$ $(\overleftrightarrow{K_n}, \hat 1)$ subgraphs. This means we now have subgraphs $(\overleftrightarrow{K_{n(n-1)}}, \hat 1)$. We thus have $\frac{n^p-n}{n(n-1)}$ $(\overleftrightarrow{K_{n(n-1)}}, \hat 1)$ subgraphs. 
        \end{enumerate}
Case 2 tells us that the $(\overleftrightarrow{K_{n(n-1)}}, \hat 1)$ subgraphs are disjoint. However, we have an edge from each nontrivial coloring to each trivial coloring since equation~(\ref{eq:6}) is always satisfied if $m'_i=0$ for all $i$. This gives the desired result.

\end{proof}

\subsection*{Acknowledgements}
Mohamed Elhamdadi was partially supported by Simons Foundation collaboration grant 712462.


\begin{thebibliography}{0}
  
\bib{BC}{article}{
   author={Basi, Jagdeep},
   author={Caprau, Carmen},
   title={Quandle coloring quivers of $(p,2)$-torus links},
   journal={J. Knot Theory Ramifications},
   volume={31},
   date={2022},
   number={9},
   pages={Paper No. 2250057, 14},
   issn={0218-2165},
   review={\MR{4475496}},
   doi={10.1142/S0218216522500572},
}

%\bib{CL}{article}{
 %  author={Chen, William Y. C.},
  % author={Louck, James D.},
   %title={The combinatorial power of the companion matrix},
   %journal={Linear Algebra Appl.},
   %volume={232},
   %date={1996},
   %pages={261--278},
   %issn={0024-3795},
   %review={\MR{1366588}},
   %doi={10.1016/0024-3795(95)90163-9},
%}

\bib{EMR}{article}{
   author={Elhamdadi, Mohamed},
   author={Macquarrie, Jennifer},
   author={Restrepo, Ricardo},
   title={Automorphism groups of quandles},
   journal={J. Algebra Appl.},
   volume={11},
   date={2012},
   number={1},
   pages={1250008, 9},
   issn={0219-4988},
   review={\MR{2900878}},
   doi={10.1142/S0219498812500089},
}
  
  \bib{EN}{book}{
  	author={Elhamdadi, Mohamed},
  	author={Nelson, Sam},
  	title={Quandles---an introduction to the algebra of knots},
  	series={Student Mathematical Library},
  	volume={74},
  	publisher={American Mathematical Society, Providence, RI},
  	date={2015},
  	pages={x+245},
  	isbn={978-1-4704-2213-4},
  	review={\MR{3379534}},
  }

  \bib{Gabriel}{article}{
   author={Gabriel, Peter},
   title={Unzerlegbare Darstellungen. I},
   language={German, with English summary},
   journal={Manuscripta Math.},
   volume={6},
   date={1972},
   pages={71--103; correction, ibid. 6 (1972), 309},
   issn={0025-2611},
   review={\MR{332887}},
   doi={10.1007/BF01298413},
}
  
%  \bib{HJ}{book}{
 %  author={Johnson, Charles R.},
 %  title={Matrix analysis},
  % edition={2},
   %publisher={Cambridge University Press, Cambridge},
   %date={2013},
   %pages={xviii+643},
   %isbn={978-0-521-54823-6},
   %review={\MR{2978290}},
%}

  \bib{Joyce}{article}{
  	author={Joyce, David},
  	title={A classifying invariant of knots, the knot quandle},
  	journal={J. Pure Appl. Algebra},
  	volume={23},
  	date={1982},
  	number={1},
  	pages={37--65},
  	issn={0022-4049},
  	review={\MR {638121 (83m:57007)}},
  	doi={10.1016/0022-4049(82)90077-9},
  }


  \bib{Matveev}{article}{
  	author={Matveev, Sergei V.},
  	title={Distributive groupoids in knot theory},
  	language={Russian},
  	journal={Mat. Sb. (N.S.)},
  	volume={119(161)},
  	date={1982},
  	number={1},
  	pages={78--88, 160},
  	issn={0368-8666},
  	review={\MR {672410 (84e:57008)}},
  }

  \bib{CN}{article}{
   author={Cho, Karina},
   author={Nelson, Sam},
   title={Quandle coloring quivers},
   journal={J. Knot Theory Ramifications},
   volume={28},
   date={2019},
   number={1},
   pages={1950001, 12},
   issn={0218-2165},
   review={\MR{3910948}},
   doi={10.1142/S0218216519500019},
}

\begin{comment}

\bib{RSS}{article}{
   author={Raundal, Hitesh}
   author={Singh, Mahender},
   author={Singh, Manpreet},
   title={Orderability of link quandles},
   journal={Proceedings of the Edinburgh Mathematical Society},
   volume={64},
   date={2021},
   number={3},
   pages={620–649},
   issn={1464-3839},
   publisher={Cambridge University Press (CUP)},
   doi={10.1017/s0013091521000419},
}
\end{comment}

 \bib{ZL}{article}{
   author={Zhou, Boxin},
   author={Liu, Ximin},
   title={Quandle coloring quivers of ($p$,3)-torus links},
   journal={J. Knot Theory Ramifications},
   volume={32},
   date={2023},
   number={3},
   pages={Paper No. 2350016, 23},
   issn={0218-2165},
   review={\MR{4581230}},
   doi={10.1142/S0218216523500165},

   
}
%\bib{4723770}{misc}{    
 %   title={An interesting way to calculate the powers of a companion matrix},    
  %  author={Apass.Jack (https://math.stackexchange.com/users/580448/apass-jack)},    
   % note={URL: https://math.stackexchange.com/q/4723770 (version: 2023-06-23)},    
    %eprint={https://math.stackexchange.com/q/4723770},    
    %organization={Mathematics Stack Exchange}  
%}
%  
%  
% 

  
  
  \end{thebibliography}
  \end{document}